\documentclass[reqno]{amsart}

 \usepackage{color}
\usepackage{amsmath,amssymb,amsfonts,amsthm,bm}
\usepackage{mathrsfs}
\usepackage{graphicx}
\usepackage{enumerate}
 \usepackage[numbers, sort&compress]{natbib}
 \usepackage[shortlabels]{enumitem}
 \newtheorem{theorem}{Theorem}[section]
 \newtheorem{corollary}[theorem]{Corollary}
 \newtheorem{lemma}[theorem]{Lemma}
 
 \theoremstyle{assumption}
   
\theoremstyle{definition}
\newtheorem{definition}[theorem]{Definition}
 \theoremstyle{remark}
\newtheorem{remark}[theorem]{Remark}
\numberwithin{equation}{section}

\newcommand{\eps}{\varepsilon}

\newcommand{\norm}[1]{ \Vert#1 \Vert}

\newcommand{\abs}[1]{\left\vert#1\right\vert}
\newcommand{\set}[1]{\left\{ #1 \right\}}

\newcommand{\inner}[1]{\left(#1\right)}
\newcommand{\biginner}[1]{\big(#1\big)}

\makeatletter

\@namedef{subjclassname@2020}{%
  \textup{2020} Mathematics Subject Classification}

\def\@startsection#1#2#3#4#5#6{%
 \if@noskipsec \leavevmode \fi
 \par \@tempskipa #4\relax
 \@afterindentfalse
 \ifdim \@tempskipa <\z@ \@tempskipa -\@tempskipa \@afterindentfalse\fi
 \if@nobreak \everypar{}\else
     \addpenalty\@secpenalty\addvspace\@tempskipa\fi
 \@ifstar{\@dblarg{\@sect{#1}{\@m}{#3}{#4}{#5}{#6}}}%
         {\@dblarg{\@sect{#1}{#2}{#3}{#4}{#5}{#6}}}%
}

\def\@settitle{%
  \bgroup
  \centering
  \vglue1cm
  \fontsize{12}{15}\fontseries{b}\selectfont
  \uppercasenonmath\@title
  \@title
  \vskip20pt plus 6pt minus 8pt
  \egroup
}

\def\@setauthors{%
  \begingroup
  \trivlist
  \centering \bfseries
 \normalsize\@topsep30\p@\relax
  \advance\@topsep by -\baselineskip
  \item\relax
  \andify\authors
 {\rmfamily\authors}%
  \endtrivlist
  \endgroup
}

\def\@setaddresses{\par
  \nobreak \begingroup
\normalsize
  \def\author##1{\nobreak\addvspace\bigskipamount}%
  \def\\{\unskip, \ignorespaces}%
  \interlinepenalty\@M
  \def\address##1##2{\begingroup
    \par\addvspace\bigskipamount\noindent
    \@ifnotempty{##1}{(\ignorespaces##1\unskip) }%
    {\ignorespaces##2}\par\endgroup}%
  \def\curraddr##1##2{\begingroup
    \@ifnotempty{##2}{\nobreak\indent{\itshape Current address}%
      \@ifnotempty{##1}{, \ignorespaces##1\unskip}\/:\space
      ##2\par}\endgroup}%
  \def\email##1##2{\begingroup
    \@ifnotempty{##2}{\nobreak\noindent{\itshape E-mail address}%
      \@ifnotempty{##1}{, \ignorespaces##1\unskip}\/: 
       ##2\par}\endgroup}%
   \def\urladdr##1##2{\begingroup
    \@ifnotempty{##2}{\nobreak\indent{\itshape URL}%
      \@ifnotempty{##1}{, \ignorespaces##1\unskip}\/:\space
      \ttfamily##2\par}\endgroup}%
  \addresses
  \endgroup
}

 \renewcommand\section{\@startsection{section}{1}{\z@}%
{27pt plus 6pt minus 8pt}{14pt plus 6pt minus 8pt}
{\center\normalfont\large\bfseries}}
\renewcommand\subsection{\@startsection{subsection}{2}{\z@}%
{27pt plus 6pt minus 8pt}{14pt plus 6pt minus 8pt}
{ \center
\normalfont\bfseries}}

\def\subsubsection{\@startsection{subsubsection}{3}%
  \z@{.5\linespacing\@plus.7\linespacing}{-.5em}%
  {\normalfont\itshape}}

 \def\@mainsize{10}\def\@ptsize{0}%
  \def\@typesizes{%
    \or{5}{6}\or{6}{7}\or{7}{8}\or{8}{9}\or{9}{11}%
    \or{10}{13}
    \or{\@xipt}{13}\or{\@xiipt}{14}\or{\@xivpt}{17}%
    \or{\@xviipt}{20}\or{\@xxpt}{24}}%
  \normalsize \linespacing=\baselineskip

\makeatother

\begin{document}

\title[Subellipticity of complex vector fields]{Subellipticity of  some  complex vector fields related to Witten Laplacian}

\author[W.-X. Li]{Wei-Xi Li}

\address[W.-X. Li]{ School of Mathematics and Statistics, Wuhan University, Wuhan 430072, China \& Hubei Key Laboratory of Computational Science, Wuhan University, Wuhan 430072, China
  }

 \email{wei-xi.li@whu.edu.cn}

 \author[C.-J. Xu]{Chao-Jiang Xu}

\address[C.-J. Xu]{Department of Mathematics, Nanjing University of Aeronautics and Astronautics, Nanjing 211106, China}

 \email{xuchaojiang@nuaa.edu.cn}

\begin{abstract}
  We consider  some system of complex vector fields   related to the semi-classical Witten Laplacian, and establish  the local subellipticity of  this system basing on   condition $(\Psi)$.  
 \end{abstract}


 \maketitle

\section{Introduction and main results}

Let  $\Omega\subset\mathbb R^n$ be a neighborhood of $0$,  and denote by  $i$    the  square root   of $-1$. We consider in this paper the following system of complex vector fields :
 \begin{equation}\label{syst+++}
  \mathcal P_j=\partial_{x_j}-i \inner{\partial_{x_j}\varphi (x)}\partial_{t},\quad
  j=1,\cdots,n,\quad (x,t)\in\Omega\times\mathbb R,
\end{equation}
where  $\varphi(x)$ is  a
{\it real-valued} function defined in a neighborhood  $\Omega$ of $0$.   This system was first studied by Treves \cite{MR0426068},   and considered therein is more general case  for $t$ varying  in $\mathbb R^m$ rather than $\mathbb R.$ 
Denote by $(\xi,\lambda)$   the Fourier variables of $(x,t)$. Then 
 the principle symbol $\sigma$ for the system $\set{\mathcal P_j}_{1\leq j\leq n}$ is   
\begin{eqnarray*}
	\sigma (x,t; \xi,\lambda)=\inner{i \xi_1+\inner{\partial_{x_1}\varphi }\lambda,\cdots, i \xi_j+\inner{\partial_{x_j}\varphi }\lambda}\in \mathbb C^n 
\end{eqnarray*}
with $ \inner{x,t; \xi,\lambda}\in T^*\inner{\Omega\times\mathbb R_t}\setminus\set{0},$ 
and thus the characteristic set is 
\begin{eqnarray*}
	\big\{(x,t;\xi,\lambda)\in T^*\inner{\Omega\times\mathbb R_t}\setminus\set{0}~\big|~\xi=0, ~\lambda\neq 0, ~\nabla\varphi(x)=0\big\}. 
\end{eqnarray*}
Since outside the characteristic set the system $\set{\mathcal P_j}_{1\leq j\leq n}$ is  elliptic,  we only need to study  the microlocal hypoellipticity  in the two components  $\set{\lambda>0}$ and  $\set{\lambda<0}$ under the assumption that 
\begin{eqnarray*}\label{gra}
	\nabla \varphi(0)=0.
\end{eqnarray*}
Note we may assume $\varphi(0)=0$ if replacing $\varphi$ by $\varphi-\varphi(0).$     By maximal hypoellipticity (in
the sense of Helffer-Nourrigat \cite{MR897103}) it means the
existence of a neighborhood $\tilde\Omega\subset \Omega $ of $0$ and a
constant $C$ such that for any $u\in C_0^\infty(\tilde\Omega\times \mathbb R)$ we have
 \begin{multline*}\label{maxhyp}
 \sum_{j=1}^n\inner{ \norm{\partial_{x_j}u}_{L^2(\mathbb R^{n+1})}
    +\norm{\inner{\partial_{x_j}\varphi} \partial_{t}u}_{L^2(\mathbb R^{n+1})} }\\
    \leq
    C \Big(\sum_{j=1}^n\norm{\mathcal P_{j}u}_{L^2(\mathbb R^{n+1})}
    +   \norm{u}_{L^2(\mathbb R^{n+1})}\Big).
  \end{multline*}
  Note the  maximal hypoellipticity will yield the following subellptic estimate   
  \begin{eqnarray*}
	\sum_{j=1}^n \norm{\partial_{x_j}u}_{L^2(\mathbb R^{n+1})}+\norm{\abs{D_t}^{\frac{1}{k+1}} u}_{L^2(\mathbb R^{n+1})}  \leq
    C  \sum_{j=1}^n\norm{\mathcal P_{j}u}_{L^2(\mathbb R^{n+1})}
    + C  \norm{u}_{L^2(\mathbb R^{n+1})},
\end{eqnarray*} 
provided $\varphi$ is finite type, i.e., for some integer $k\geq 1$,
\begin{eqnarray*}
	\sum_{\abs{\alpha}\leq k+1}\abs{\partial_x^\alpha\varphi(0)}\neq 0.
\end{eqnarray*}      
Thus the subellipticity is in some sense intermediate  between the maximal hypoellipticity and the local hypoellipticity.

We first recall the related works concerned with the subellipticity and  hypoellipticity as well as the links between these properties. 
 For the system of real vector fields,  it is well-known (\cite{Hormander67, Kohn78, RothschildStein76}) that the H\"ormander's  bracket condition is sufficient to obtain the subellipticity  in terms of  the Lie algebra generated by the real vector fields.  Moreover for real vector fields with analytic coefficients,  Derridj's work \cite{MR0407433} implies that the  H\"ormander's condition and thus subellipticity  is also a necessary condition to obtain hypoellipticity.    But the situation is quite different  in the setting of  complex vector fields;   Kohn \cite{MR2183286} constructed a counterexample,  showing we may  lose the subellipticity even though  the H\"ormander's bracket condition is fulfilled.      The hypoellipticity for the system \eqref{syst+++} was  initiated  by Treves  \cite{MR0426068} , and then continued  by   Maire \cite{MR567778}.  The standard $L^2$  subellipticity was investigated in Derridj \cite{MR3603885, MR2298786} and  a series of works \cite{MR2868553, MR2885111,MR2471931}    of Derridj-Helffer,  where the crucial tool  is to find suitable escaping curves.   So far it remain unclear  for the links between hypoellipticity and subellipticity for   systems of complex vector fields.   In fact we may ask the same question as in the case of real vector fields,  wether hypoellipticity implies subellipticity for system \eqref{syst+++}  under the assumption that $\varphi$ is analytic in $\Omega$.   This is  true for $n=1,$ but may be false when $n>1$ due to the work  \cite{MR2276079} of Journ{\'e} and  Tr{\'e}preau,  where they showed for some specific polynomial potential $\varphi$ that the system is hypoelliptic without being subelliptic.  For specific  quasihomogeneous  analytic potential functions  Derridj-Helffer  \cite{MR2885111} proved  hypoellipticity system \eqref{syst+++} is   indeed subelliptic for the dimension $n=2$.     Finally we mention  the criterion of Helffer-Nourrigat   \cite{ MR897103},  which is based on   the nilpotent group techniques   to characterize  the maximal hypoellipticity  for   general systems of one order pseudodifferential operators; see also Nourrigat's work \cite{MR1044428, MR1098615} for further development and Helffer-Nier'   lecture \cite{HelfferNier05}  for the  application of  Helffer-Nourrigat's criterion to Witten Laplacian.  
 
 Observe for the system $\set{\mathcal P_j}_{1\leq j\leq n} $,   we may perform partial Fourier transform with respect to $t$,  and  study  the  microhypoellipticity,  in the two directions $\lambda>0$  and $\lambda<0.$  Indeed we only need consider without loss of generality the   microhypoellipticity in the positive direction $\lambda>0,$ since the other direction $\lambda<0$ can be treated similarly by replacing $\varphi$ by $-\varphi.$   Consider the resulting system as follows after taking partial Fourier transform for $t\in\mathbb R$:
\begin{equation}\label{system2}
 \mathcal L_j=\partial_{x_j}+\lambda \inner{\partial_{x_j}\varphi},\quad
  j=1,\cdots,n.
\end{equation}
By maximal microhypoellipticity at $0$ in the positive direction in $\lambda>0$,  it means the existence of a positive number $\lambda_0>0,$   a
constant $C>0$  and  a neighborhood $\tilde\Omega\subset \Omega$ of $0,$  such that
\begin{multline}\label{hypo}
    \forall~ \lambda\geq \lambda_0,~\forall~u\in
  C_0^\infty(\tilde\Omega),\\
  \sum_{j=1}^n \inner{\norm{\partial_{x_j}u}_{L^2}
    +\norm{\lambda \inner{\partial_{x_j}\varphi} u}_{L^2}}\leq
    C\Big( \sum_{j=1}^n\norm{ \mathcal L_ju}_{L^2}
    +\norm{u}_{L^2}\Big),
  \end{multline}
where and throughout the paper we denote     $\norm{\cdot}_{L^2(\mathbb R^n)}$   by  $\norm{\cdot}_{L^2}$ for short  if no confusion occurs.  We remark  the  operators  defined in \eqref{system2}   is  closely related to the semi-classical Witten Laplacian
  \begin{eqnarray*} 
  \triangle_{ \lambda \varphi}^{(0)} =-\triangle_x +\lambda^{2} \abs{\partial_x \varphi}^2-\lambda \triangle_x   \varphi
  \end{eqnarray*}
  with $\lambda^{-1}$ the semi-classical parameter,  by the relationship
  \begin{eqnarray*}
  	\sum_{j=1}^n \norm{ \mathcal L_j u}^2_{L^2}=\inner{\triangle_{ \lambda \varphi}^{(0)} u, \ u}_{L^2},
  \end{eqnarray*}
  where $\inner{\cdot,  \cdot}_{L^2}$ stands for the inner product in $L^2(\mathbb R^n).$   Helffer-Nier \cite{HelfferNier05}  conjectured $ \triangle_{ \lambda \varphi}^{(0)} $  is subelliptic near $0$ if  $\varphi$ is analytic and has no local minimum near $0,$ and  this still remains open so far.   Note  \eqref{hypo} is a local estimate concerning the sharp  regularity near $0\in \mathbb R^n$ for $\lambda>0$,  and we have also its global   counterpart,  which is of independent interest for  analyzing  the semi-classical  lower bound of  Witten Laplacian.    We refer to Helffer-Nier's work \cite{HelfferNier05}  for the  detailed presentation on the topic of global maximal hypoellipticity and its application to the spectral analysis on Witten Laplacian; see also  the first author's  work \cite{MR3775156} for the extension to the non-polynomial potentials.  

In this work we aim to investigate the hypoellipticity by virtue of  condition $(\Psi)$. Recall condition   $(\Psi)$ was 
initially  introduced to study  the solvability of a pseudo-differential    
operator $p^w$ with complex-valued symbol $p_1+ip_2$; we refer to \cite{MR781537,MR2599384} for the comprehensive discussion on  condition  $(\Psi)$.   For a specific scalar  operator  
\begin{eqnarray*}
	\partial_{s}+  \mu(s),
\end{eqnarray*}
with $\mu$ real-valued function defined in an interval $I\subset\mathbb R$, 
condition $(\Psi)$ means the following change-of-sign condition:
\begin{eqnarray*}
	s\rightarrow \mu(s)\  \textrm{does not change sign from + to} - \textrm{as } s \ \textrm{increases in}\  I. 
\end{eqnarray*}
Moreover we say $\partial_{s}+ \mu(s) $ satisfies condition $( \overline \Psi)$ if $\partial_{s}- \mu(s)$ satisfies condition $(\Psi)$, that is,
\begin{eqnarray*}
	s\rightarrow \mu(s)\ \textrm{does not change sign from}  -  \textrm{to + as } s \ \textrm{increases in}\  I. 
\end{eqnarray*}
For the scalar operator $\partial_{s}-i \mu(s)\partial_{t}$,  the subellipticity is well explored by   condition $(\overline\Psi)$;  for instance if $\mu$ is finite type and
has no sign change from $-$ to $+$ as $s$ increases, then it is 
subelliptic (c.f.
\cite{MR0290201} and  \cite[Chapter 27]{MR781536}).  

Note there is no condition $(\Psi)$ for systems, in particular for the system $\mathcal L_j, 1\leq j\leq n,$ defined by \eqref{system2}.   For clear presentation we  consider the case of $n=2$ and our argument may be applied to more general case of $n\geq 2.$  
In what follows
we use $(x,y)$   instead of
$(x_1,x_2)$ as variables.  Then we can rewrite $\mathcal L_j, j=1,2,$ as
 \begin{equation}\label{syst}
 \left\{
 \begin{aligned}
& \mathcal L_1=\partial_{x}+\lambda \inner{\partial_{x}\varphi},\\
&\mathcal L_2=\partial_{y}+\lambda \inner{\partial_{y}\varphi}.
\end{aligned}
\right.
\end{equation}

To state our main result we first introduce the following assumption  which, roughly
speaking, says that the set of the points $(x,y)$ at which both of the functions  
$s\rightarrow \partial_x\varphi (s,y)$ and   $s\rightarrow \partial_y\varphi (x,s)$ change
sign from $-$ to $+$ for increasing $s$, has measure $0$.

\begin{definition}
[{\bf Assumption} $\boldsymbol{H_1(\alpha)}$ and $\boldsymbol{H_2(\alpha)}$] \label{def}
 Let $\alpha>0$ and let $\varphi$ be defined in a neighborhood $\Omega$ of $0$. Write  $\Omega=\Omega_x\times\Omega_y$  with $\Omega_x,\Omega_y$ being  the projections of $\Omega$ onto $x$ and $y$ axes, respectively.          We say $\varphi$ satisfies  Assumption $\boldsymbol{H_1(\alpha)}$ in $\Omega$ if   
    there exists  a family of disjoint open sets $\omega_{j} $  and   $C^1$-functions $r_j$,  $ 
    j=1,2,\cdots, N_1$ with $N_1$ an integer such that the following properties are fulfilled by $\omega_j$ and $r_j$.   
    \begin{enumerate}[(i)]
    	\item Denote by $\bar\omega_j$   the closure of $\omega_j$.       Then
    	  \begin{eqnarray*}
    	 \bigcup_{1\leq j\leq N_1} \ \bar \omega_j =\bar\Omega_y.
    \end{eqnarray*}  
    	\item Given $y\in\ \omega_j$ for $1\leq j\leq N_1$,   as $s$ increases in
  $\Omega_x$ the function
  $
    s\longmapsto\partial_x\varphi(s,y)
  $ either has no sign change from $-$ to $+$,
  or
  only changes sign from $-$ to $+$ at the point $r_{j}(y)$.  In the latter we suppose additionally that
  \begin{equation*}
s\longmapsto (\partial_y\varphi)(r_{j}(y), s)
  \end{equation*}
  doesn't change sign from $-$ to $+$ as $s$ increases.
  \item For any $1\leq j\leq N_1$ we have  $\big|\frac{d r_{j}}{dy}\big| \geq \alpha$
  in $\omega_{j}$.
    \end{enumerate}
Moreover Assumption $\boldsymbol{H_2(\alpha)}$ can be defined  symmetrically   by a family of pairs $(U_j,\tau_j)$, $1\leq j\leq N_2$, which satisfy similar properties as that of  $(\omega_j, r_j)$ such that for each $x\in U_j$,
 \begin{equation*}
s\longmapsto (\partial_x\varphi)(s,\tau_{j}(x))
  \end{equation*}
  does not change sign from $-$ to $+$ as $s$ increases if $s\longmapsto \partial_y\varphi(x,s)$ changes sign from $-$ to $+$ at point $\tau_j(x)$.
\end{definition}

\begin{theorem}\label{MainT2}
 Let $\alpha>0$ be given and  denote by
  $P_m$ the set of  polynomials with degree  $\leq m$. Then there exists
  a number $\delta$ depending only on  the degree $m$, and two
  constants $\lambda_0$ and $C$ depending only on $m$ and $\alpha$,  such that
  the following estimate 
  \begin{equation*}
\forall \ u\in C_0^\infty(]-\delta,\delta[ \ \times\  ]-\delta,\delta[) ,\quad    \norm{\partial_{x}u}_{L^2}+\norm{\lambda  (\partial_x\varphi ) u}_{L^2}\leq
    C\sum_{1\leq j\leq 2}\norm{\mathcal L_{j}u}_{L^2}
  \end{equation*}
  holds for all $\lambda\geq \lambda_0$ and
  for all $\varphi\in P_m$  satisfying  Assumption
 $\boldsymbol{H_1(\alpha)}$ in  the  neighborhood $]-\delta,\delta[ \ \times\  ]-\delta,\delta[$ of $0$.  Symmetrically 
 the   estimate 
  \begin{equation*}
\forall \ u\in C_0^\infty(]-\delta,\delta[ \ \times\  ]-\delta,\delta[) , \quad    \norm{\partial_{y}u}_{L^2}+\norm{\lambda  (\partial_y\varphi ) u}_{L^2}\leq
    C\sum_{1\leq j\leq 2}\norm{\mathcal L_{j}u}_{L^2}
  \end{equation*}
  holds for all $\lambda>\lambda_0$ and
  for all $\varphi\in P_m$  satisfying  Assumption
 $\boldsymbol{H_2(\alpha)}$ in $]-\delta,\delta[ \ \times\  ]-\delta,\delta[$. 
\end{theorem}

\begin{corollary}
If $\varphi\in  P_m$ satisfies  Assumption
 $\boldsymbol{H_1(\alpha)}$ in a neighborhood of $0$ and the condition that
 \begin{eqnarray*}
 \sum_{0\leq j\leq k}	\abs{\partial_x^{k+1} \varphi(0)}\neq 0
 \end{eqnarray*}
 for some $k$. Then system \eqref{syst} is (micro)subelliptic at positive direction, that is, there exists a constant $C>0,$ a neighborhood $\tilde\Omega$ of $0$ and a positive number $\lambda_0$  such that  
\begin{eqnarray*}
    \forall~ \lambda\geq \lambda_0,~\forall~u\in
  C_0^\infty(\tilde\Omega),\quad 
    \norm{\lambda^{1/(k+1)} u}_{L^2}\leq
    C \sum_{1\leq j\leq 2}\norm{\mathcal L_j u}_{L^2}.  \end{eqnarray*}    
	If $\varphi\in  P_m$ satisfies  Assumptions
 $\boldsymbol{H_1(\alpha)}$  and $\boldsymbol{H_2(\alpha)}$ in a neighborhood of $0$   then we have the following maximal (micro)hypoellipticity:  for any $\lambda\geq \lambda_0$ and any $u\in
  C_0^\infty(\tilde\Omega),$ 
 \begin{multline*}
      \norm{\partial_{x}u}_{L^2}+  \norm{\partial_{y}u}_{L^2}
    +\norm{\lambda \inner{\partial_{x}\varphi} u}_{L^2}
    +\norm{\lambda \inner{\partial_{y}\varphi} u}_{L^2}\leq
    C \sum_{1\leq j\leq 2}\norm{\mathcal L_j u}_{L^2}.
    \end{multline*}
\end{corollary}

\begin{remark} 
  This improves a result of Nourrigat (see Theorem 2.1 of \cite{MR1098615}) by allowing the sign change from $-$ to $+$
  for increasing variable.  The above results may hold for general $\varphi$ rather than polynomial case,  following the argument presented in \cite{MR3775156}. 
\end{remark}



We would end up the introduction by an example of Maire (cf. \cite[Example 1.2]{MR567778}).
   These are the functions   
  $\varphi$ defined by
  $$\varphi(x,y)=x^{2\ell+1}-xy^2,\quad \ell\in\mathbb Z_+\setminus \set{0}.$$
  Then the system \eqref{syst} has the form of
  \begin{equation*}
    \left\{
    \begin{array}{lll}
      \mathcal L_1=\partial_{x}+\lambda \inner{(2\ell+1)x^{2\ell}-y^2}\\
      \mathcal L_2=\partial_{y}-2\lambda xy.
    \end{array}
    \right.
  \end{equation*}
  Direct verification  shows  that  Assumption $\boldsymbol{H_1(\alpha)}$ in $\mathbb R^2$ is fulfilled if we
  we can take $\omega_{1}=]-\infty,0[$ and
  $r_{1}(y)=\inner{-\frac{y}{\sqrt{2\ell+1}}}^{1\over\ell}$, and $\omega_{2}=]0,\infty[$ and
  $r_{2}(y)=\inner{\frac{y}{\sqrt{ 2\ell+1}}}^{1\over\ell}$. Then it follows Theorem \ref{MainT2} that  there exists two constants $C>0,\lambda_0\geq 0$ such that
 for any $u\in C_0^\infty(\mathbb R^2)$  and for all $\lambda\geq\lambda_0$,
  \begin{equation*}
    \norm{\partial_{x}u}_{L^2}+\norm{\lambda (\partial_{x}\varphi)u}_{L^2}\leq
    C\norm{\mathcal L_{1}u}_{L^2}+C\norm{\mathcal L_{2}u}_{L^2}.
  \end{equation*}
  As a result 
  \begin{equation*}
    \norm{ \lambda ^{\frac{1}{2\ell+1}}u}_{L^2}\leq
    C\norm{\mathcal L_{1}u}_{L^2}+C\norm{\mathcal L_{2}u}_{L^2}.
  \end{equation*}
  And hence the system $\mathcal L_j, 1\leq j\leq 2,$  is  subelliptic  with exponent $\frac{1}{2\ell+1}$.   We give here a new proof for the Maire's example basing on condition $(\Psi)$.    As proven by Helffer-Nier (see Chapter 11 of \cite{HelfferNier05}), the exponent
  $\frac{1}{2\ell+1}$ is optimal and can not be improved. This implies
  the system $(\mathcal L_1,\mathcal L_2)$ is not maximally
  hypoelliptic when $\ell>1$, otherwise a better exponent
  $\sigma=\frac{1}{3}$ would be deduced.

\section{Proof of Theorem \ref{MainT2}}

In this part we prove the main result.  Firstly we introduce some notations to be used frequently later. Let
$\varphi(x,y)$ be a polynomial of degree $m$. Then we can write  
$\partial_x\varphi(x,y)$ and $\partial_y\varphi(x,y)$ as the following
forms
\begin{equation}\label{10101520}
  \partial_x\varphi(x,y)=\sum_{j=0}^{\ell} a_j(y)x^j,  ~
  \textrm{and} ~\partial_y\varphi(x,y)=\sum_{j=0}^{d}
  b_j(x)y^j,
\end{equation}
where $\ell,d\leq m$ and $a_j(y)$, $b_j(x)$ are polynomials of $y$ and $x$
respectively, such that the leading coefficients $a_{\ell}\not\equiv
0$ and $b_{d}\not\equiv0$. For each pair $(x,y)$ we set
 \begin{equation*}
 \left\{
 \begin{aligned}
    & M_{1}(x,y)=\sum_{j=0}^\ell
     \abs{\lambda \partial_x^{j+1} \varphi(x,y)}^{\frac{1}{j+1}},\\
     &M_{2}(x,y)=\sum_{j=0}^d\abs{\lambda \partial_y^{j+1}
 \varphi(x,y)}^{\frac{1}{j+1}},
     \end{aligned}
     \right.
  \end{equation*}
  and  set
  \begin{equation}\label{deg}
 	G(x,y)=\sum_{1\leq i+j\leq m}\abs{\lambda \partial_x^{i}
  \partial_y^j\varphi(x,y)}^{\frac{1}{ i+j}}.
  \end{equation}
Let $\mathcal N_1$, $\mathcal N_2$ be
  defined by
  \begin{equation}\label{n1n2}
  \mathcal N_1=\set{y\in\mathbb R ; ~M_1(x,y)\neq 0~\textrm{for
  all~} x}, \quad
   \mathcal N_2=\set{x\in\mathbb R ; ~M_2(x,y)\neq 0~\textrm{for all~} y} .
   \end{equation}
  We remark that  $\mathcal N_1$ and  $\mathcal N_2$ are
  dense in $\mathbb R$, since $\mathcal N_1$ and $\mathcal N_2$ include,respectively, the
  sets $\set{y ; ~a_{\ell}(y)\neq 0}$ and $\set{x ; ~b_{d}(x)\neq0}$ which are dense in $\mathbb R$.

To simplify the notation we will use the capital letter $C$
to denote some generic constant  that may vary from line to line and   depend only on the number $\alpha$  in Definition \ref{def} and the degree $m$ of the polynomial $\varphi$, but are independent of $\lambda$.

\begin{lemma}\label{legu}
  Let $\delta>0$ be a given number, and let $I$ be a subset of
  $\mathbb R$. Then there exists a constant $C$ such that for any $u\in C_0^\infty\inner{]-\delta,  \delta[\,\times ]-\delta,  \delta[}$ we have
  \begin{equation}\label{10101701}
    \norm{G  u}_{L^2(I\times ]-\delta,\delta[)}\leq C
    \sum_{1\leq j\leq 2}\norm{\mathcal L_{j} u}_{L^2}
    +C \inner{\norm{M_{1}  u}_{L^2}+
    \norm{M_{2}  u}_{L^2(I\times ]-\delta,\delta[)}},
  \end{equation}
  recalling $G$ is defined  by \eqref{deg}.
\end{lemma}

\begin{proof}
    If $I=]-\delta,\delta[$ the estimate \eqref{10101701} is a straightforward consequence
    of Baker-Campbell-Hausdorff formula (see e.g. \cite[Lemma 4.14]{MR2987644}). 
    Next we treat the general
    case, following the argument presented in \cite{MR2987644}.  In this proof we always let $u\in
    C_0^\infty\inner{]-\delta,  \delta[\,\times ]-\delta,  \delta[}$.
    Direct verification  shows 
    \begin{equation}\label{wfd}
    	\norm{\partial_x u}_{L^2}^2+\norm{\lambda(\partial_x\varphi) u}_{L^2}^2=
    \norm{\mathcal L_1 u}_{L^2}^2+
    \inner{ \lambda (\partial_x^2\varphi) u,\ u}_{L^2}^2.
    \end{equation}
    Then
     $$\norm{\partial_x u}_{L^2}^2\leq
    \norm{\mathcal L_1 u}_{L^2}^2+
    \norm{|\lambda \partial_x^2\varphi|^{1\over
    2}u}_{L^2}^2.$$ Furthermore since the function
  $y\longmapsto u(x,y)$ with $x\in I$ fixed  has compact support in the interval
  $]-\delta,\delta[$, then
  \begin{equation*}
    \norm{\partial_y u}_{L^2(I\times ]-\delta,\delta[)}^2\leq
    \norm{\mathcal L_2 u}_{L^2(I\times ]-\delta,\delta[)}^2+
    \norm{|\lambda \partial_y^2\varphi|^{1/2}u}_{L^2(I\times ]-\delta,\delta[)}^2.
  \end{equation*}
  It follows  that
  \begin{equation*}
    \norm{X u}_{L^2}^2+
    \norm{Y u}_{L^2}^2\leq
    \sum_{1\leq j\leq 2}\norm{\mathcal L_j u}_{L^2}^2+2
    \norm{M_{1}u}_{L^2}^2
    +\norm{M_{2}u}_{L^2(I\times ]-\delta,\delta[)}^2,
  \end{equation*}
  where the operators $X$ and $Y$ are defined by $X=\frac{\partial_x}{i}$  and
  $Y=\chi_I(x)\frac{\partial_y}{i}$ with $\chi_I(x)$ being the characteristic function of the set
  $I$.
  Thus the desired estimate will follow if there exists a constant $C$ such that for
  any integer $q\geq 1$  one has
  \begin{equation}\label{pq0}
    \sum_{p\geq 0}
    \norm{\abs{A_{p,q}}^{\frac{1}{p+q+1}} u}_{L^2}\leq C\inner{
    \norm{X u}_{L^2}+
    \norm{Y u}_{L^2}
    +\norm{ \lambda (\partial_x\varphi) u}_{L^2}},
  \end{equation}
  where $A_{p,q}$ are  purely imaginary-valued functions defined iteratively  by
  $$
  A_{0,q}= i^{q-1}  [\underbrace{X,~ [X,\cdots [X}_{q},~ \lambda \partial_x\varphi]~]\cdots],\quad
  A_{p,q}= i    [Y,~A_{p-1,q}] \textrm{~ for ~}p\geq 1.
  $$
  Here we denote by $[P,\ Q]$ the commutator of  two operators $P$ and $Q$,  which is defined by $
		[P,\ Q]=PQ-QP
$. 
  Firstly note that for
   any $q\geq 1$ we have
  \begin{equation}\label{p0}
    \norm{\abs{A_{0,q}}^{\frac{1}{q+1}} u}_{L^2}\leq C\inner{
    \norm{\partial_x u}_{L^2}+
    \norm{ \lambda (\partial_x\varphi )u}_{L^2}+
    \norm{u}_{L^2}}.
  \end{equation}
  This can be deduced directly from Baker-Campbell-Hausdorff formula  (see e.g. \cite[Lemma 4.14]{MR2987644}). Next we will
  treat the case when $p\geq 1$.  Baker-Campbell-Hausdorff formula implies
  \begin{equation*}
  \begin{split}
    \exp\inner{tY}\exp\inner{t^{p+q}A_{p-1,q}}\exp\inner{-tY}\exp\inner{-t^{p+q}A_{p-1,q}}\\=
    \exp\biginner{t^{p+q+1}A_{p,q}+\sum_{j=1}^{m}c_j t^{p+q+j+1}A_{p+j, q}},
  \end{split}
  \end{equation*}
  where $c_j$ are a family of constants. As a result we use induction on $p$ to conclude
  that there exist two integers $n_1$ and $n_2$ depending only on $p$,
  and a sequence $(\eps_1,\eps_2,\cdots,\eps_{n_1})$ with $\eps_j=\pm
  1$, such that for each $p\geq 1$ one has
  \begin{equation*}
  \begin{split}
    \exp\inner{t^{p+q+1}A_{p,q}}
    &=\prod_{ j=1}^{n_2} Z_j  
    \exp\biginner{\sum_{k=1}^m c_k t^{p+q+k+1}A_{p+k, q}}\\
    &=\prod_{ j=1}^{n_2} Z_j  
    \prod_{k=1}^m\exp\inner{c_k t^{p+q+k+1}A_{p+k, q}}
  \end{split}
  \end{equation*}
   with $Z_j\in\set{\exp\inner{\eps_{1} tY}, \cdots,
  \exp\inner{\eps_{n_1} tY},
   \exp\inner{t^{q+1}A_{0,q}}}$. The last equality holds because $A_{p,q}$ are functions and hence
  commutative. Then $$\exp\inner{t^{p+q+1}A_{p,q}}-{\rm Id}$$   with ${\rm Id}$ the identity
  operator, can be rewritten as
  \begin{equation*}
  \begin{split}
    &\prod_{k=1}^{n_2} Z_k\sum_{i=1}^m\prod_{ j=1}^{i-1}
    \exp\biginner{c_j t^{q+j+1}A_{p+j, q}}
    \big(\exp\biginner{c_i t^{q+i+1}A_{p+i, q}}-{\rm
    Id}\big) \\
    &+\sum_{i=1}^{n_2}\prod_{ j=1}^{i-1} Z_j\inner{Z_i-{\rm Id}},
  \end{split}
  \end{equation*}
  which, together with fact that the norms  of operators $Z_j$ are smaller than $1$,
  implies
  \begin{multline} \label{10101705}
    \norm{\exp\inner{t^{p+q+1}A_{p,q}}u-u}_{L^2}^2\leq C 
    \sum_{1\leq k\leq m}\norm{\exp\inner{c_k t^{k+q+1}A_{p+k,
    q}}u-u}_{L^2}^2\\
    +C\sum_{\eps_I\in\big\{\eps_1,\cdots,\eps_{n_1}\big\}} \norm{\exp\inner{\eps_{I} tY}u-u}_{L^2}^2
     +C\norm{\exp\inner{t^{q+1}A_{0,q}}u-u}_{L^2}^2.
  \end{multline}
  Since the functions $i  A_{p,q}$ are real-valued and hence
  self-adjoint as operators acting on $L^2$. So by the
  spectral theorem the norm
  $\norm{|i  A_{p,q} |^{\frac{1}{p+q+1}}u}_{L^2}$  is equivalent to
  \[
     (p+q+1)^{-1}\int_0^\infty
     \norm{\exp\inner{
     t^{p+q+1} A_{p,q}}u-u}_{L^2}^2  \frac{dt}{t^{3}}.
  \]
  As a result, integrating both sides of \eqref{10101705} with respect to $t$ with the measure $t^{-3}dt$
  gives that
  \begin{equation*}
    \norm{\abs{A_{p,q}}^{\frac{1}{p+q+1}} u}_{L^2}\leq C
    \sum_{k=1}^m\norm{\abs{A_{p+k,q}}^{\frac{1}{p+k+q+1}} u}_{L^2}\\
    +C\norm{Y u}_{L^2}
    +C\norm{\abs{A_{0,q}}^{\frac{1}{q+1}}u}_{L^2}.
  \end{equation*}
  Using reverse induction starting from $p=m$ and combing \eqref{p0},
  we get the desired estimate \eqref{pq0}, completing the proof.
\end{proof}

\begin{lemma}
  There exist two  constants $C_*\geq 1$ and $r_0\leq 1$ depending only on $m$,  such that for any $y\in \mathcal N_1$ and $x\in \mathcal N_2$  with $\mathcal N_j$ defined by \eqref{n1n2}, we have
  \begin{equation}\label{10101401}
    \forall~x_*,\tilde x\in \mathbb R, ~\abs{x_*-\tilde x}
    M_{1}(x_*,y)<r_0 \Longrightarrow C_*^{-1}  \leq \frac{ M_{1}(x_*,y)}{M_{1}(\tilde
    x,y)}\leq C_* ,
  \end{equation}
  \begin{equation}\label{10101401+}
    \forall~y_*,\tilde y\in \mathbb R, ~\abs{y_*-\tilde y} M_{2}(x,y_*)<r_0
    \Longrightarrow  C_*^{-1} \leq \frac{M_{2}(x,y_*)}{M_{2}( x,\tilde y)}\leq C_*.
  \end{equation}
 And for any $x\in\mathbb R,$
  \begin{equation}\label{10101401++}
    \forall~y_*,\tilde y\in \mathbb R, ~\abs{y_*-\tilde y} G(x,y_*)<r_0
    \Longrightarrow  C_*^{-1} \leq \frac{ G (x,y_*)}{G ( x,\tilde y)}\leq C_*.
  \end{equation}
\end{lemma}

\begin{proof}
  We  firstly show \eqref{10101401}, and claim that      for  any $y\in \mathcal N_1$
    we have, with each $  0\leq j\leq \ell$ with  $\ell$ the integer given in \eqref{10101520},
  \begin{equation}\label{10101402}
 \forall\ \abs{x_*-\tilde x} M_{1}(x_*,y)<1,\quad \abs{\lambda \partial_x^{j+1} \varphi(\tilde x,y)}  \leq
    (\ell+2) M_{1}(x_*,y)^{j+1}.
  \end{equation}
     In fact, by
 Taylor'expansion
  $$
    \partial_x^{j+1} \varphi(\tilde x,y)=
    \partial_x^{j+1}\varphi(x_*,y)+\sum_{i=0}^{\ell-j}
    \frac{\partial_x^{i+j+1}\varphi(x_*,y)}{i!}(\tilde
    x-x_*)^i,
  $$
 we have
  $$
    \abs{\lambda \partial_x^{j+1} \varphi(\tilde x,y)} \leq
    M_{1}(x_*,y)^{j+1}+ \sum_{i=0}^{\ell-j} M_{1}(x_*,y)^{i+j+1}\abs{\tilde
    x-x_*}^i,
  $$ which yields the assertion
  \eqref{10101402}. As a result it follows from \eqref{10101402} that 
  \begin{equation}\label{ies}
  \forall\  x_*, \tilde x\in\mathbb R\  \textrm{with}\,\abs{x_*-\tilde x} M_{1}(x_*,y)<1,\quad	M_{1}(\tilde x,y)\leq (\ell+2)^2M_{1}(x_*,y).
  \end{equation}
  Now for any $x_*$ and $\tilde x$ with $ \abs{x_*-\tilde x} M_{1}(x_*,y)<r_0\leq 1$,  we use \eqref{ies} to get 
  $$\abs{x_*-\tilde x} M_{1}(\tilde x,y)\leq (\ell+2)^2 \abs{x_*-\tilde x} M_{1}(x_*,y)<(\ell+2)^2r_0.$$
  As a result, if  we choose $r_0$ sufficiently small such that  $(\ell+2)^2r_0<1$ then we use \eqref{ies} to conclude
  \begin{eqnarray*}
  	M_{1}( x_*,y)\leq (\ell+2)^2M_{1}(\tilde x,y).
  \end{eqnarray*}
 We have proven \eqref{10101401}.  
  The inequalities \eqref{10101401+} and \eqref{10101401++} can be deduced similarly. This completes the proof.
\end{proof}

\begin{lemma}\label{MH}
  Let $P_m$ be the set of all polynomials with degree  
  $\leq m$, and let  $r\in C^1\inner{\omega; ~\mathbb
  R}$ be a given function defined in $\omega$.
  Then there
  exists a number $\delta>0$ and a constant  $C>0$, both depending only on $m$,
   such that the
  following estimate 
  \begin{equation}\label{10101406}
    \norm{M_2u}_{L^2\inner{I\times ]-\delta,\delta[}}\leq
    C\norm{\mathcal L_2 u}_{L^2}
  \end{equation}
  holds for all $u\in C_0^\infty\inner{]-\delta,\delta[\,\times\,]-\delta,\delta[}$ and for all $\varphi\in P_m$ satisfying
  that for each $s\in\omega$ the function
  $y\longmapsto \partial_y\varphi\inner{r(s),  y}$
  has no sign change from $-$ to $+$  for increasing $y$, where  $I$ stands for the range of $r$, i.e.,  $I=r(\omega)$.
\end{lemma}

\begin{remark}
  If $r=Id$ then this is just a local version of the result of Nourrigat (see Theorem 2.1 of \cite{MR1098615}).
\end{remark}

To prove the above lemma we need the following
\begin{lemma} 
	[Lemmas 3.1.1-3.1.3 of  \cite{MR2599384}] \label{lemsig} Let $[a, b]\subset\mathbb R$ be a given interval, and let  $\phi(s)\in C^0\inner{[a,b];\,\mathbb R}$ such that  $\phi$ does not change sign from $-$ to $+$ as $s$ increases in $[a,b]$. Then for any $v\in C^1\inner{[a,b]}$ with $v(a)=0$ we have  
	\begin{eqnarray*}
		\max_{s\in[a,b]}\abs{v(s)}^2+2\int_a^b\abs{\lambda\phi(s)}\abs{v(s)}^2ds\leq \abs{v(b)}^2+2\int_a^b\Big|\big(\frac{d}{ds}+\lambda\phi\big)v(s)\Big|\abs{v(s)}ds,
	\end{eqnarray*}
	where $\lambda>0$ is a given number.  In particular, for any for any $v\in C_0^1\inner{[a,b]}$  we have  
	\begin{eqnarray*}
		\max_{s\in[a,b]}\abs{v(s)}^2+2\int_a^b\abs{\lambda\phi(s)}\abs{v(s)}^2ds\leq  2\int_a^b\Big|\big(\frac{d}{ds}+\lambda\phi\big)v(s)\Big|\abs{v(s)}ds.
	\end{eqnarray*} 
Moreover if 
	\begin{eqnarray*}
		\min_{s\in[a,b]}\Big|\frac{d^k\phi}{ds^k}(s)\Big|\geq c_*
	\end{eqnarray*}
	for some   $k\in\mathbb Z_+$ and some constant $c_*>0$, then for any $v\in C_0^1\inner{[a,b]}$, we have
	\begin{eqnarray*}
\lambda^{\frac{1}{k+1}}\int_a^b  \abs{v(s)}^2ds \leq  C\Big(\max_{s\in[a,b]}\abs{v(s)}^2+\int_a^b\abs{\lambda\phi(s)}\abs{v(s)}^2ds\Big),
	\end{eqnarray*}
	with $C$ a constant depending only on $k$ and $c_*$ above. 
\end{lemma}
\begin{proof}
	The just follows from the argument for proving     \cite [Lemma 3.1.1]{MR2599384} and the assertions in   \cite [Lemmas 3.1.2-3.1.3]{MR2599384}.  
	\end{proof}

\begin{proof}[Proof of Lemma \ref{MH}] Recall  $\mathcal N_2$ is defined by \eqref{n1n2}. 
  Since $r(\omega)\setminus\mathcal N_2$ has measure $0$ then we may assume without loss of generality that  $r(\omega)\subset\mathcal
  N_2$. 
  For $x_0=r(s_0)\in r(\omega)\subset \mathcal
  N_2$,  $y_0\in \mathbb R$ and $\varphi\in P_m$
  we define a new function
  $\zeta(x_0,\cdot)$ by setting
  \[
    \zeta(x_0,y)=\frac{ (\partial_y \varphi)\inner{x_0,
    y_0+\frac{y}{M_{2}(x_0,y_0)}}}{M_{2}(x_0,y_0)}.
  \]
  Then we use Taylor's expansion to write
     \begin{eqnarray*}
  	    \zeta(x_0,y)=\sum_{0\leq j\leq d}\frac{  \partial_y^{j+1} \varphi\inner{x_0,
    y_0}}{M_{2}(x_0,y_0)^{j+1}}\frac{y^j}{j!},
  \end{eqnarray*}
 where the integer $d$ is given in \eqref{10101520}.  By direct verification, 
  \begin{equation*} 
  \forall\ 0\leq j \leq d,\quad\frac{ \abs{  \partial_y^{j+1} \varphi\inner{x_0,
    y_0}}}{M_{2}(x_0,y_0)^{j+1}}\leq 1	,
  \end{equation*}
  which yields
  \begin{equation}\label{ues}
  \forall\ 0\leq j \leq d,\quad\abs{ \partial_y^{j} \zeta\inner{x_0,
    0}}\leq 1.
  \end{equation}
  On the other hand, we can find an integer $k$ with $0\leq k\leq d\leq m$ such that
  \begin{eqnarray*}
  	\abs{ \partial_y^{k+1} \varphi\inner{x_0,
    y_0}}^{1/(k+1)} =\max_{0\leq j\leq d}\abs{ \partial_y^{j+1} \varphi\inner{x_0,
    y_0}}^{1/(j+1)} .
  \end{eqnarray*}
  Thus
  \begin{equation*}
    \abs{\inner{\partial_y^k\zeta}
    \inner{x_0,0}}=\frac{ \partial_y^{k+1} \varphi\inner{x_0,
    y_0}}{M_{2}(x_0,y_0)^{k+1}}\geq c_0 
  \end{equation*}
   for some constant $c_0>0$ depending only on $m$.
This, with \eqref{ues}, implies 
    $$\abs{\inner{\partial_y^k\zeta}
    \inner{x_0,y}}\geq  c_0 -\sum_{1\leq j\leq d-k} \abs{y}^{j}. $$
  As a result one can find a number $\delta>0$  depending only on
  $m$,  such that
  \begin{equation}\label{10101405}
    \inf_{y\in]-\delta,\delta[} \abs{\inner{\partial_y^k\zeta}
    \inner{x_0,y}}\geq c_0/2.
  \end{equation}
  Moreover by assumption $\partial_y\varphi(x_0,y)$ has no sign change from $-$ to $+$ for increasing $y$.
  Then the function
  $\zeta(x_0,y)$ doesn't change sign from $-$ to $+$ as $y$
  increases in $]-\delta,\delta[$. As a result we apply Lemma \ref{lemsig} with $\phi=\zeta$ to conclude    for any $v\in C_0^\infty(]-\delta,\delta[)$ we have
  \begin{equation*}\label{10101501}
    \lambda^{\frac{1}{k+1}}\int_{-\delta}^\delta \abs{v(y)}^2dy \leq
     C \int_{-\delta}^\delta
    \abs{(\partial_y+\lambda \zeta(x_0,y))v(y)} ^2dy+C\int_{-\delta}^\delta\abs{v(y)}^2
     d y.
  \end{equation*}
   Inspired by \cite{MR781536},  for any function $v\in C^\infty\inner{[-\delta,\delta]}$  having no compact
  support in $]-\delta,\delta[$, we can introduce a cut-off function
  $\chi(y)$ supported in the interval $]-\delta,\delta[$  and equal to $1$ in
  $]-\delta/2,\delta/2[$.
  Then applying the above inequality to the function $\chi(y)  v(y)$ gives
  that, for any $v\in C^\infty([-\delta,\delta])$,
    \begin{equation}\label{10101502}
    \lambda^{\frac{1}{k+1}}\int_{-{\delta\over2}}^{\delta\over 2}
     \abs{v(y)}^2 d y
    \leq
     C \int_{-\delta}^\delta
    \abs{\inner{\partial_y + \lambda \zeta\inner{x_0,y}}v(y)}^2 
   dy+C \int_{-\delta}^\delta \abs{v(y)}^2 d y.
  \end{equation}
 Now for any given  $u\in C_0^\infty\inner{]-\delta,\delta[\,\times\,]-\delta,\delta[}$ we define 
 $v\in C^\infty([-\delta,\delta])$ by setting
  \[
    v(y)=u\Big(x_0, y_0+\frac{y}{M_{2}(x_0,y_0)}\Big).
  \]
  For such a function $v$ we can verify that
  \[
    \big(\partial_y+ \lambda \zeta(x_0,y_0)\big)v (y)= \frac{(\mathcal L_2 u)\inner{x_0, y_0+y/M_{2}(x_0,y_0)}}{M_{2}(x_0,y_0)}.
  \]
  As a result, applying \eqref{10101502} to the function $v$ above   and then using the change of variables, we get
  \begin{multline*}
     \lambda^{\frac{1}{k+1}}\int_{I_{y_0}}
  M_{2}(x_0,y_0)^3 \abs{u(x_0,y)}^2d y\\     \leq
    C\int_{J_{y_0}} M_{2}(x_0,y_0)
    \abs{(\mathcal L_2 u)(x_0,y)}^2
     d y  + C \int_{J_{y_0}}  M_{2}(x_0,y_0)^3 \abs{  u(x_0,y)}^2
     d y,
  \end{multline*}
  where 
  \begin{eqnarray*}
  I_{y_0}=\big\{y ; ~ \abs{y-y_0}M_{2}(x_0,y_0)<\delta/2\big\},\quad J_{y_0}=\big\{y ; ~ \abs{y-y_0}M_2(
  x_0,y_0)<\delta\big\}.
 \end{eqnarray*} 
  By \eqref{10101401+} we see
  $C_*^{-1}\leq  M_2(x_0,y_0)/M_{2}(x_0,y)\leq C_*$
  when $y\in J_{y_0}$,  shrinking $\delta$ if necessary. Then we may replay $M_{2}(x_0,y_0)$ by
  $M_{2}(x_0,y)$ in the above integrands, that is,
\begin{multline*}\label{m2es}
     \lambda^{\frac{1}{k+1}}\int_{I_{y_0}}
  M_{2}(x_0,y)^3 \abs{u(x_0,y)}^2d y\\     \leq
    C\int_{J_{y_0}} M_{2}(x_0,y)
    \abs{(\mathcal L_2 u)(x_0,y)}^2
     d y  + C \int_{J_{y_0}}  M_{2}(x_0,y)^3 \abs{  u(x_0,y)}^2
     d y.
  \end{multline*}
  Thus taking integration with respect to $y_0\in\mathbb R$ on the both sides of above inequality and  observing  \begin{eqnarray*}
  	\set{y ; ~
  \abs{y-y_0}M_2(x_0,y)<\delta/(2C_*)}\subset I_{y_0}, \quad    J_{y_0}\subset \set{y ; ~
  \abs{y-y_0}M_2(x_0,y)<C_*\delta },
  \end{eqnarray*}
   we obtain, applying Fubini's theorem,
  \begin{multline*}
    \lambda^{\frac{1}{k+1}}\int_{-\delta}^{\delta}
    \abs{M_{2}(x_0,  y)u(x_0,y)}^2d y \leq
    C\int_{-\delta}^{\delta}
    \abs{(\mathcal L_2 u)(x_0,y)}^2dy \\
     + C \int_{-\delta}^{\delta}  \abs{M_{2}(x_0, y)u(x_0,y)}^2
     d y.
  \end{multline*}
  Note that the constant $C$ is independent of $\lambda$, and hence we
  can choose $\lambda$ large enough such that $\lambda^{\frac{1}{k+1}}\geq
  2C$. This implies
  \begin{equation*}
    \lambda^{\frac{1}{k+1}}\int_{-\delta}^{\delta}
    \abs{M_{2}(x_0,  y)u(x_0,y)}^2d y \leq
    C\int_{-\delta}^{\delta}
    \abs{(\mathcal L_2 u)(x_0,y)}^2dy
  \end{equation*}
  Now integrating  both sides with respect to $x_0\in r(\omega)$ gives the
  desired estimate \eqref{10101406}.
  The proof is hence completed.
\end{proof}

\begin{lemma}
 Let $\delta, \alpha>0$   be two given numbers with $\delta$ small sufficiently and let  $r(y)\in C^1\inner{\omega; ~\mathbb
  R}$ such that  $\big|\frac{dr}{dy}(y)\big|\geq \alpha$ for any 
  $y\in \omega$. Suppose $\varphi\in P_m$ satisfies
  that for each $y_0\in\omega$ the function
  $y\longmapsto (\partial_y\varphi)\inner{r(y_0),  y}$
  has no sign change from $-$ to $+$ for increasing $y$.
  Then  the following esitmate
  \begin{equation}\label{10101505}
    \int_{\omega} G(r(y),y)\abs{u(r(y),y)}^2d y
    \leq C  \sum_{1\leq j\leq 2}\norm{\mathcal L_j u}_{L^2}^2+C
    \norm{M_{1}u}_{L^2}^2
  \end{equation}
  holds for all $u\in C_0^\infty\inner{]-\delta,\delta[\,\times\,]-\delta,\delta[}$.  Recall $G=\sum_{i+j\geq 1}\abs{  \lambda\partial_x^{i}
 \partial_y^{j}\varphi}^{\frac{1}{i+j}}$.
\end{lemma}

\begin{proof}
This proof is quite similar as  Lemma \ref{MH}.  Let $y_0\in\omega $ be given and
we define a new function
  $y\longmapsto \eta(r(y_0),y)$ by setting
  \[
    \eta(r(y_0),y)=\frac{  (\partial_y \varphi )\inner{r(y_0),
    y_0+y/G(r(y_0), y_0)}}{G(r(y_0), 
    y_0)}.
  \]
  Then by the assumption, the function $y\longmapsto \eta\inner{r(y_0), y}$ does not change sign from $-$ to $+$ as   $y$ increases in $]-\delta,\delta[$.  Then we apply Lemma \ref{lemsig} to conclude, for any $v\in C_0^\infty\inner{]-\delta,\delta[}$ we have 
  \begin{eqnarray*}
 \max_{y\in ]-\delta,\delta[}    \abs{v(y)}^2  \leq
   2 \int_{-\delta}^\delta
    \abs{(\partial_y+\lambda \eta(r(y_0),y))v(y)} \abs{v(y)} dy.
     \end{eqnarray*}
     Thus applying  the above inequality to $\chi(y)  v(y)$, with $\chi\in C_0^{\infty}\inner{]-\delta,\delta[}$ and $\chi\equiv 1$ on $[-\delta/2,\delta/2]$, we obtain, for any $v\in C^\infty([-\delta,\delta])$
  \begin{eqnarray*}
 \abs{v(0)}^2\leq \max_{y\in ]-\frac{\delta}{2},\frac{\delta}{2}[}    \abs{v(y)}^2  \leq
   C \int_{-\delta}^\delta
    \abs{(\partial_y+\lambda \eta(r(y_0),y))v(y)}^2dy+C\int_{-\delta}^\delta \abs{v(y)}^2 dy.
     \end{eqnarray*}
  As a result for any $u\in C_0^\infty\inner{]-\delta,\delta[\,\times\,]-\delta,\delta[}$, we apply the above estimate to  
  \[
    v(y):=u\Big(r(y_0), y_0+\frac{y}{G(r(y_0),y_0)}\Big);
  \]
  this yields, by virtue of the change of variables,
\begin{multline*}
    G(r(y_0),y_0)^2\abs{u(r(y_0),y_0)}^2\leq
    C  \int_{K_{y_0}}
    \abs{(\mathcal L_2 u)(r(y_0),y)}^2dy\\
    +C  \int_{K_{y_0}} G(r(y_0),y_0)^2\abs{u(r(y_0),y)}^2
     d y,
  \end{multline*}
   where $K_{y_0}=\set{y ; ~ \abs{y-y_0}G(
  r(y_0),y_0)<\delta}$.
  It follows from \eqref{10101401++}  that
  \begin{eqnarray*}
  	\forall\ y\in K_{y_0},\quad C_*^{-1}G(r(y_0),y)\leq  G(r(y_0),y_0)\leq C_*G(r(y_0),y).
  \end{eqnarray*}
   Then we may replay $G(r(y_0),y_0)$ by
  $G(r(y_0),y)$ in the above integrands. 
  This gives
\begin{multline*}
    G(r(y_0), y_0)^2\abs{u(r(y_0),y_0)}^2\\ \leq
    C  \int_{-\delta}^\delta 
    \abs{(\mathcal L_2 u)(r(y_0),y)}^2dy+C\int_{-\delta}^\delta  G(r(y_0), y) ^2\abs{ u(r(y_0),y)}^2
     d y.
  \end{multline*}
  Integrating both sides
  with respect to $y_0\in\omega$ gives
  \begin{equation*}
  \begin{aligned}
   & \int_{\omega} G(r(y_0),y_0)^2\abs{u(r(y_0),y_0)}^2d y_0\\
   &\leq C  \int_{\omega} \int_{-\delta}^\delta \inner{
    \abs{(\mathcal L_2 u)(r(y_0),y)}^2 +   G(r(y_0), y)^2\abs{ u(r(y_0),y)}^2}
     d ydy_0\\
   &\leq C   
    \norm{ \mathcal L_2 u}_{L^2}^2 +  C\norm{G     u }_{L^2(I\times]-\delta,\delta[)}^2    \end{aligned}
  \end{equation*} 
  with $I=r\inner{\omega}$, 
  the last inequality using the change of variables $r(y_0)\rightarrow x$ and the fact that $\abs{r'(y)}\geq \alpha$ for all $y\in\omega$.    Moreover  we apply Lemma \ref{legu} and Lemma \ref{MH} to conclude
  \begin{eqnarray*}
  	\begin{aligned}
  		\norm{G     u }_{L^2(I\times]-\delta,\delta[)} & \leq  C
    \sum_{1\leq j\leq 2}\norm{\mathcal L_{j} u}_{L^2}
    +C \inner{\norm{M_{1}  u}_{L^2}+
    \norm{M_{2}  u}_{L^2(I\times ]-\delta.\delta[)}}\\& \leq  C
    \sum_{1\leq j\leq 2}\norm{\mathcal L_{j} u}_{L^2}
    +C  \norm{M_{1}  u}_{L^2}. 
  	\end{aligned}
  \end{eqnarray*}
   Combining the above inequalities  gives the desired estimate \eqref{10101505}.
  The proof is completed.
  \end{proof}

\begin{proof}[Proof of Theorem \ref{MainT2}]
  Let $\varphi\in P_m$ satisfy Assumption $\boldsymbol{H_1(\alpha)}$ and let $\omega_j, r_j, 1\leq j\leq N_1,$ be given therein.  For any $y_0\in \omega_j \cap \mathcal N_1$ with $1\leq j\leq N_1$ and
  $x_0\in\mathbb R$ we define a new
  function $x\longmapsto \xi(x,y_0)$ by setting
  \[
     \xi(x, y_0)=\frac{  (\partial_x \varphi)\inner{x_0+\frac{x}{M_{1}(x_0,y_0)},
     y_0}}{M_{1}(x_0,y_0)}.
  \]
  Similar to the proof of \eqref{10101405}, we may find an integer $k$ such that
  \begin{equation}\label{10101508}
    \inf_{x\in]-\delta,\delta[} \abs{\inner{\partial_x^k\xi}
    \inner{x,y_0}}\geq  c_0  /2
  \end{equation}
  with $c_0,\delta$ depending only on $m$. 
  Denote
  $$\mathcal A(y_0)=\set{x\in\mathbb R;  ~\abs{x-r(y_0)}M_{1}(x,y_0)<\delta}.$$
  We have two cases that either $x_0\in\mathcal A(y_0)$ or $x_0\not\in\mathcal
  A(y_0)$.

\noindent {\it Case (a)}.  We  firstly consider the case when $x_0\in\mathcal A(y_0)$, that is,
\begin{equation}\label{eqdel}
	\abs{r(y_0)-x_0}M_{1}(x_0,
  y_0)\leq \delta.
\end{equation}
Since   $\varphi\in P_m$ satisfies Assumption $\boldsymbol{H_1(\alpha)}$ then  the function $ x\longmapsto
  \partial_x\varphi\inner{x,y_0}$ only changes sign from $-$ to $+$
  at $r(y_0)$  for increasing $x$  and the function
  $y\longmapsto \partial_y\varphi\inner{r(y_0),  y} $
  has no sign change from $-$ to $+$ for increasing $y$.
    This with \eqref{eqdel}   implies that  $x\longmapsto \xi(x,y_0)$  only changes sign
  from $-$ to $+$ at the point $\inner{r(y_0)-x_0}M_{1}(x_0,
  y_0)$ for $x$ increases in $]-\delta,\delta[$.   
  Then we apply the  first assertion in Lemma \ref{lemsig} with $[a, b]=[-\delta, \inner{r(y_0)-x_0}M_{1}(x_0,
  y_0) ]$ or $[a, b]=[\inner{r(y_0)-x_0}M_{1}(x_0,
  y_0), \delta]$, to conclude 
  for any $v\in C_0^\infty(-]\delta,\delta[)$,
  \begin{equation*}\label{10101512}
  \begin{aligned}
   &  \sup_{x\in]-\delta,\delta[}\abs{v(x)}^2+2\int_{-\delta}^\delta
     \abs{\lambda\xi(x,y_0)}\abs{v(x)}^2 d x\\
     &\leq
    2\abs{v \big(\inner{r(y_0)-x_0}M_{1}(x_0,y_0)\big)}^2 
     +4\int_{-\delta}^\delta
    \abs{(\partial_x+\lambda\xi(x,y_0))v(x)} \abs{v(x)}
     d x.
  \end{aligned}
  \end{equation*}
  which, with \eqref{10101508}
   and the last assertion in Lemma \ref{lemsig}, yields, for any $v\in C_0^\infty(]-\delta,\delta[)$, 
   \begin{equation*}\label{10101512}
  \begin{aligned}
   &  \lambda^{\frac{1}{k+1}}\int_{-\delta}^\delta
       \abs{v(x)}^2 d x\\
     &\leq
    C\abs{v \big(\inner{r(y_0)-x_0}M_{1}(x_0,y_0)\big)}^2 
     +C\int_{-\delta}^\delta
    \abs{(\partial_x+\lambda\xi(x,y_0))v(x)} \abs{v(x)}
     d x.
  \end{aligned}  
  \end{equation*}
  Then we follow the argument for proving  Lemma \ref{MH} to conclude,  for any $v\in C^\infty([-\delta,\delta])$
  \begin{equation*}\label{10101512}
  \begin{aligned}
   &  \lambda^{\frac{1}{k+1}}\int_{-\delta/2}^{\delta/2}
       \abs{v(x)}^2 d x\\
     &\leq
    C\abs{v \big(\inner{r(y_0)-x_0}M_{1}(x_0,y_0)\big)}^2 
     +C\int_{-\delta}^\delta
    \abs{(\partial_x+\lambda\xi(x,y_0))v(x)} \abs{v(x)}
     d x.
  \end{aligned}  
  \end{equation*}
  Now for any $u\in C_0^\infty(]-\delta,\delta[\,\times\,]-\delta,\delta[) $ we apply the above estimate to
  \begin{eqnarray*}
  	v(x):=u\Big(x_0+\frac{x}{M_1(x_0,y_0)}, y_0\Big);
  \end{eqnarray*}
  this gives,
  for any  $u\in C_0^\infty(]-\delta,\delta[\,\times\,]-\delta,\delta[) $,
  \begin{equation}\label{mla}
  \begin{split}
    &\lambda^{\frac{1}{k+1}}\int_{\tilde I_{x_0}}
   M_{1}(x_0,y_0)^3\abs{u(x,y_0)}^2d x \leq C
   M_{1}(x_0,y_0)^3\abs{u\inner{r(y_0),y_0}}^2\\
    &+C\int_{\tilde J_{x_0}}M_{1}(x_0,y_0) 
    \abs{\mathcal L_1 u(x,y_0)} ^2  
     d x+C \int_{\tilde J_{x_0}} M_{1}(x_0 ,y_0)^3\abs{u(x,y_0)}^2
     d x,
  \end{split}
  \end{equation}
  where 
  \begin{eqnarray*}
  	\tilde I_{x_0}=\big\{x ; ~ \abs{x-x_0}M_{1}(x_0,y_0)<\frac{\delta}{2}\big\},\ \tilde  J_{x_0}=\big\{x ; \  \abs{x-x_0}M_{1}(
  x_0,y_0)<\delta\big\}.
  \end{eqnarray*}
It follows from \eqref{10101401} that
  $C_*^{-1}\leq  M_1(x_0,y_0)/M_{1}(x,y)\leq C_*$
  when   $x\in \tilde J_{x_0}$.     
  Moreover in view of \eqref{eqdel}, 
  \begin{equation*}\label{10101509}
   C_*^{-1} M_{1}(r(y_0), y_0) \leq M_{1}(x_0, y_0)\leq C_* M_{1}(r(y_0), y_0). 
  \end{equation*}
  Consequently  we   replay $M_{1}(x_0,y_0)$ by
  $M_{1}(x,y_0)$ in the above integrands and and  by
  $M_{1}(r(y_0),y_0)$ in the first term on the right side, and then  integrate both side with respect to $x_0\in \mathcal
  A(y_0)$; this gives
  \begin{equation}\label{10101510}
  \begin{aligned}
   & \lambda^{\frac{1}{k+1}}\int_{\mathcal A(y_0)}\int_{\tilde I_{x_0}}
   M_{1}(x,y_0)^3\abs{u(x,y_0)}^2d x dx_0  \leq C
   M_{1}(r(y_0),y_0)^2\abs{u\inner{r(y_0),y_0}}^2\\
    &\quad +C \int_{\mathcal A(y_0)}\int_{\tilde J_{x_0}}
 \Big(M_{1}(x,y_0)   \abs{\mathcal L_1 u(x,y_0)}^2 
  +
    M_{1}(x
    ,y_0)^3\abs{u(x,y_0)}^2 \Big)
     d x dx_0.
     \end{aligned}
  \end{equation}
  Here we have used the fact that $\mathcal A(y_0)\subset\set{x ; ~\abs{x-r(y_0)} M_{1}(r(y_0),y_0)<C_*\delta}$
  and hence the measure of $\mathcal A(y_0)$ is less than
  $2C_*\delta/M_{1}(r(y_0),y_0)$.

 \noindent {\it Case (b)}.   Next we consider the case that $x_0\not\in\mathcal
  A(y_0)$. Then  $$\abs{r(y_0)-x_0}M_{1}(x_0,
  y_0)\geq \delta.$$  This implies $\xi(x,y_0)$  has no sign change
  from $-$ to $+$ as $x$ increases in $]-\delta,\delta[$. Then the
  estimate \eqref{mla} still holds with the absence of the first term on the
  right side. Then repeating the above arguments, we can
  show that, with $x_0$ varying in $\mathcal A^c(y_0)\stackrel{\rm def}{=}\mathbb R\setminus \mathcal
  A( y_0 )$,
  \begin{equation}\label{10101515}
  \begin{aligned}
   & \lambda^{\frac{1}{k+1}}\int_{\mathcal A^c(y_0)}\int_{\tilde I_{x_0}}
   M_{1}(x,y_0)^3\abs{u(x,y_0)}^2d x dx_0  \\
    &\leq C \int_{\mathcal A^c(y_0)}\int_{\tilde J_{x_0}}
 \Big(M_{1}(x,y_0)   \abs{\mathcal L_1 u(x,y_0)}^2 
  +
    M_{1}(x
    ,y_0)^3\abs{u(x,y_0)}^2 \Big)
     d x dx_0.
     \end{aligned}
  \end{equation}

\noindent {\it Completeness of the proof of Theorem \ref{MainT2}}.  In conclusion we  take  sum of \eqref{10101510} and \eqref{10101515}, to get
  \begin{equation*}
  \begin{split}
    &\lambda^{\frac{1}{k+1}}\int_{\mathbb R}\int_{\tilde I_{x_0}}
   M_{1}(x,y_0)^3\abs{u(x,y_0)}^2d x dx_0 \leq C
   M_{1}(r(y_0),y_0)^2\abs{u\inner{r(y_0),y_0}}^2\\
    &+C \int_{\mathbb R}\int_{\tilde J_{x_0}}
    \inner{M_{1}(x,y_0)\abs{\mathcal L_1 u(x,y_0)}^2  +M_{1}(x
    ,y_0)^3\abs{u(x,y_0)}^2}
     d x dx_0,
  \end{split}
  \end{equation*}
  which, with Fubini'theorem as well as the facts  that \begin{eqnarray*}
  \Big\{x ; ~ \abs{x-x_0}M_{1}(x,y_0)<\frac{\delta}{2C_*}\Big\}\subset	\tilde I_{x_0},\ \tilde  J_{x_0}\subset\big\{x ; \  \abs{x-x_0}M_{1}(
  x,y_0)<C_*\delta\big\},
  \end{eqnarray*} 
   implies
    \begin{equation*}
  \begin{split}
    &\lambda^{\frac{1}{k+1}}\int_{-\delta}^{\delta}
   M_{1}(x,y_0)^2\abs{u(x,y_0)}^2d x \leq C
   M_{1}(r(y_0),y_0)^2\abs{u\inner{r(y_0),y_0}}^2\\
    &+C \int_{-\delta}^{\delta}
    \inner{\abs{\mathcal L_1 u(x,y_0)}^2  +M_{1}(x
    ,y_0)^2\abs{u(x,y_0)}^2}
     d x.
  \end{split}
  \end{equation*}
  Observe the above inequality holds for all $y_0\in \omega_j \cap \mathcal N_1$ with $\mathcal N_1$ dense in $\omega_j$. Then after integration with respect to
  $y_0$ we obtain
  \begin{eqnarray*}
  \begin{aligned}
    &\lambda^{\frac{1}{k+1}}\norm{
    M_{1}  u}_{L^2\inner{]-\delta,\delta[\,\times \omega_j}}\\
   & \leq C
    \int_{\omega_j}M_{1}(r(y),y)\abs{u\inner{r(y),y}}^2  d y+C
    \norm{\mathcal L_1 u}_{L^2}^2+C\norm{
   M_{1}  u}_{L^2}^2\\
    &\leq C \sum_{1\leq j\leq 2}
    \norm{\mathcal L_j u}_{L^2}^2+C\norm{
   M_{1}  u}_{L^2}^2,
   \end{aligned}
  \end{eqnarray*}
  the last inequality following from \eqref{10101505} since $M_1\leq G$.  Observe $]-\delta,\delta[ \subset \bigcup_{1\leq j\leq N_1} \bar \omega_j$ with $\omega_j $ disjoint, and thus
  \begin{eqnarray*}
  \begin{aligned}
     \lambda^{\frac{1}{k+1}}\norm{
    M_{1}  u}_{L^2} \leq C \sum_{1\leq j\leq 2}
    \norm{\mathcal L_j u}_{L^2}^2+C\norm{
   M_{1}  u}_{L^2}^2.
   \end{aligned}
  \end{eqnarray*}
   Now we choose $\lambda_0$  such that $\lambda_0^{\frac{1}{k+1}}=4C$; this gives  
  \begin{eqnarray*}
 \forall\ u\in C_0^\infty(]-\delta,\delta[\,\times\,]-\delta,\delta[),\  \forall\ \lambda\geq \lambda_0, \quad \norm{
    M_{1}  u}_{L^2} \leq C \sum_{1\leq j\leq 2}
    \norm{\mathcal L_j u}_{L^2}^2.
  \end{eqnarray*}
 In view of \eqref{wfd} we have for any $u\in C_0^\infty(]-\delta,\delta[\,\times\,]-\delta,\delta[)$,
 \begin{eqnarray*}
 	\norm{\partial_xu}_{L^2}^2+\norm{\lambda(\partial_x\varphi)u}_{L^2}^2\leq  \norm{\mathcal L_1 u}_{L^2}^2+\norm{M_1 u}_{L^2}^2,
 \end{eqnarray*}
  and thus the desired estimate follows.  The proof of Theorem \ref{MainT2} is  completed.
\end{proof}

{\bf Acknowledgments} 
 The research of the first author was supported by NSFC (Nos. 11871054,11961160716, 11771342) and the Fundamental Research Funds for the Central Universities(No.2042020kf0210).


\end{document}